\documentclass{svjour3}                     % onecolumn (standard format)

\smartqed  % flush right qed marks, e.g. at end of proof
\UseRawInputEncoding
\usepackage{bbm}
\usepackage{mathrsfs}
\usepackage{amsfonts}
\usepackage{stmaryrd}
\usepackage{graphicx}
\usepackage{subfigure}

\usepackage{bm}% bold math
\usepackage{amsmath}
\usepackage{amssymb}

\usepackage{amstext}
\usepackage{amsbsy}
\usepackage{setspace}
\usepackage{algorithm}
\usepackage{algorithmic}
\usepackage[figuresright]{rotating}
\usepackage{epstopdf}
\usepackage{makecell}
\usepackage{color}

\numberwithin{equation}{section}
\numberwithin{theorem}{section}
\numberwithin{lemma}{section}
\numberwithin{remark}{section}
\usepackage{mathptmx}

\newtheorem{ass}{Assumption}

\usepackage{mathptmx}      % use Times fonts if available on your TeX system
\usepackage[justification=centering]{caption}
\date{}
\begin{document}

\title{High-accuracy time discretization of stochastic fractional diffusion equation}


\author{Xing Liu
}


\institute{
Xing Liu \at
Bigdata Modeling and Intelligent Computing research institute, Hubei University of Education, Wuhan 430205, People's Republic of China. \email{2718826413@qq.com}
             % Tel.: +123-45-678910\\
%              Fax: +123-45-678910\\
             % \email{wsqlzu@gmail.com}
 %  \emph{Present address:} of F. Author  %  if needed
%\and
%Weihua Deng\at School of Mathematics and Statistics, Gansu Key Laboratory of Applied Mathematics and Complex
%Systems, Lanzhou University, Lanzhou 730000, People's Republic of China.
%\email{dengwh@lzu.edu.cn}
}

\maketitle

\begin{abstract}	
A high-accuracy time discretization is discussed to numerically solve the nonlinear fractional diffusion equation forced by a space-time white noise. The main purpose of this paper is to improve the temporal convergence rate by modifying the semi-implicit Euler scheme. The solution of the equation is only H\"older continuous in time, which is disadvantageous to improve the temporal convergence rate. Firstly, the system is transformed into an equivalent form having better regularity than the original one in time. But the regularity of nonlinear term remains unchanged. Then, combining Lagrange mean value theorem and independent increments of Brownian motion leads to a higher accuracy discretization of nonlinear term which ensures the implementation of the proposed time discretization scheme without loss of convergence rate. Our scheme can improve the convergence rate from ${\min\{\frac{\gamma}{2\alpha},\frac{1}{2}\}}$ to ${\min\{\frac{\gamma}{\alpha},1\}}$ in the sense of mean-squared $L^2$-norm. The theoretical error estimates are confirmed by extensive numerical experiments.
\\
\\
\keywords{high-accuracy time discretization; modifying the semi-implicit Euler scheme; the regularity of nonlinear term; mean-squared $L^2$-norm}
%\noindent{\bf AMS} 26A33, 35R11, 65M60, 65M12.

\subclass{26A33 \and 65M60 \and 65L20 \and65C30}
\end{abstract}

\section{Introduction}
The anomalous diffusion phenomena are usually modeled macroscopically by partial differential equations (PDFs), describing the distribution of particle concentration in space, at each moment \cite{Deng2020,Dybiec2009}. In addition to the inherent laws of particle motion, sometimes the external stochastic disturbances can also affect the distribution of particle concentration. Thus, a noisy force term is included in the model problems \cite{Chow2007}. In this paper, we discuss the time discretization of the following model

 \begin{equation}\label{eq:1.1}
\frac{\partial u(x,t)}{\partial t}=-(-\mathrm{\Delta})^\alpha u(x,t)+f\left(u(x,t)\right)+\dot{B}(x,t)
\end{equation}
with the initial and homogeneous Dirichlet boundary condition given by
\begin{eqnarray*}
 u(x,0)=u_0(x),\quad x\in D,\\
 u(x,t)=0,\quad (x,t)\in \partial D\times[0,T],
  \end{eqnarray*}
and $\alpha \in (0,1)$. Here $\dot{B}(x,t)=\frac{\partial B(x,t)}{\partial t}$ is the white noise; ${B}(x,t)$ represents an infinite dimensional Brownian motion. The role of nonlinear term $f$ is to produce or destroy particles. Let $-\mathrm{\Delta}$ be the
infinitesimal generator of killed Brownian motion; then $(-\mathrm{\Delta})^\alpha$ is the infinitesimal generator of the
subordinated killed Brownian motion \cite{Song2003}. It shows that  \cite{Nochetto2015,Song2003}  if $\{(\lambda_i,\phi_i)\}^\infty_{i=1}$ are the eigenpairs of $-\mathrm{\Delta}$, then $\{(\lambda^\alpha_i,\phi_i)\}^\infty_{i=1}$ are the eigenpairs of $(-\mathrm{\Delta})^\alpha$, i.e.,
 \begin{equation}\label{eq:1.03}
 \left\{
\begin{array}{cc}
 -\mathrm{\Delta}\phi_i=\lambda_i\phi_i, \quad &\mathrm{in} \ D,\\
  \quad \phi_i=0,\quad &\mathrm{on} \ \partial D,
 \end{array}
 \right.
 \end{equation}
and
\begin{equation}\label{eq:1.04}
 \left\{
\begin{array}{cc}
(-\mathrm{\Delta})^\alpha\phi_i =\lambda^\alpha_i\phi_i, \quad &\mathrm{in} \ D,\\
  \quad \phi_i=0,\quad &\mathrm{on} \  \partial D.
 \end{array}
 \right.
 \end{equation}

Recently, many efforts have been made in theoretical and numerical studies of the stochastic PDEs. For example, the regularity of the solution for stochastic PDEs was studied \cite{Choi2021,Hu2017,Song2019}.
%the random periodic solutions and stationary solutions of certain stochastic PDEs were derived by using Feynman-Kac formulas \cite{Song2019}. In \cite{Hu2017}, the authors have studied the existence and uniqueness of the solution for nonlinear stochastic heat equation driven by a Gaussian noise which contains white noise in time and fractional noise in the space. Existence, uniqueness, and maximal H\"older regularity of a strong solution were discussed for stochastic PDEs with a colored noise \cite{Choi2021}.
In \cite{Brehier2014}, using the semi-implicit Euler scheme has studied the time discretization of a parabolic stochastic PDE. Based on an exponential scheme in time and spectral Galerkin methods in space, the posteriori error estimate of the numerical method for the stochastic Allen-Cahn equation was derived \cite{Blomker2019}. A Markov jump process approximation of stochastic PDEs has been introduced \cite{Bou2018}. The authors of \cite{Liu2020} established a numerical scheme whose mean-square convergence order was $1/4$ in time. In \cite{Chen2019}, based on H\"older continuity of the solution, the mean-square convergence order $1/2$ was derived in the temporal direction for stochastic Maxwell equations. In \cite{Cui2019}, the authors used the spectral splitting Crank-Nicolson scheme to obtain the approximation for the solutions of stochastic nonlinear Schr\"odinger equations, and the temporal convergence order $1/2$ in the mean-square sense was obtained. More recently, the authors of \cite{LiuX2020} discussed a fully discrete scheme for the fractional diffusion equation forced by a tempered fractional Gaussian noise. The existing results show that the temporal convergence order typically depends on the H\"older regularity of solution in time. For instance, these authors of \cite{Chen2019,Gunzburger2019,Liu2020} prove that
\begin{equation*}
 \left\|u(t_n)-u_n\right\|_{L^2(D,U)}\le C(t-s)^{\min\{\frac{\gamma}{2\alpha},\frac{1}{2}\}}\left(1+\left\|u_0\right\|_{L^2\left(D,\dot{U}^\gamma\right)}\right),
\end{equation*}
when the solution in mean-square $L^2$-norm is only $\min\{\frac{\gamma}{2\alpha},\frac{1}{2}\}$-H\"older continuous, i.e.,
\begin{equation*}
 \left\|u(t)-u(s)\right\|_{L^2(D,U)}\le C(t-s)^{\min\{\frac{\gamma}{2\alpha},\frac{1}{2}\}}\left(1+\left\|u_0\right\|_{L^2\left(D,\dot{U}^\gamma\right)}\right).
\end{equation*}
Here, $u_n$ denotes the numerical solution of stochastic PDEs forced by a space-time white noise at time $t_n$.

With the above introduction, an interesting question arises as to whether it is possible to design a temporal discretization which strongly converges with a rate faster than $\min\{\frac{\gamma}{2\alpha},\frac{1}{2}\}$ in problem \eqref{eq:1.1}. In this paper, we provide a positive answer to this question by modifying the semi-implicit Euler scheme. We first transform Eq. \eqref{eq:1.1} into an equivalent form having better regularity than the original one in time by using Ornstein-Uhlenbeck process, i.e.,
 \begin{equation*}
\mathrm{d} z(t)+A^\alpha  z(t)\mathrm{d}t=f\left(u(t)\right)\mathrm{d}t,
\end{equation*}
where $u(t)=z(t)+\int^t_0S(t-s)\mathrm{d}B(s)$ and $S(t-s)=\mathrm{e}^{-(t-s)A^\alpha}$. Then using the semi-implicit Euler scheme discretize the above equation:
 \begin{equation*}
 z_{n+1}-z_{n}+(t_{n+1}-t_{n}) A^\alpha  z_{n+1}=(t_{n+1}-t_{n}) f\left(u_{n}\right),
\end{equation*}
where $u_n=z_n+\int^{t_n}_0S(t_n-s)\mathrm{d}B(s)$. In fact, the convergence rate remains unchanged because the accuracy of approximation for nonlinear term $f$ is not improved. Thanks to independent increment for $B(t)$ and Lagrange mean value theorem, we can design a higher accuracy discretization of $f$. The proposed scheme possesses the following convergence rate under the mean-square $L^2$-norm:
\begin{equation*}
 \left\|u(t_n)-u_n\right\|_{L^2(D,U)}\le C(t-s)^{\min\{\frac{\gamma}{\alpha},1\}}\left(1+\left\|u_0\right\|_{L^2\left(D,\dot{U}^\gamma\right)}\right).
\end{equation*}

This paper is organized as follows. In the next section, we introduce some preliminaries, including definitions, assumptions and properties of Brownian motion. In Sect. \ref{sec:3}, we prove the H\"older regularity of the mild solution $z(t)$ for the equivalent form \eqref{eq:3.2} in mean-squared $L^2$-norm.  In Sect. \ref{sec:4}, we modify the semi-implicit Euler method to obtain a high order time discretization of \eqref{eq:1.1}; and the convergence rate estimate for the proposed discrete scheme is derived. The numerical experiments are performed in Sect. \ref{sec:5}. We end the paper with some discussions in Sect. \ref{sec:6}.

\section{Notations and preliminaries} \label{sec:2}
In this section, we give some notations, upper bound estimate of eigenvalues for $-\mathrm{\Delta}$, and definitions of function space, which are commonly used in the paper to derive the strong convergence rate of time discretization.

We take $U=L^2(D;\mathbb{R})$ to denote the space of real-valued 2-times integrable functions with inner product $\langle\cdot,\cdot\rangle$ and norm $\|\cdot\|=\langle\cdot,\cdot\rangle^{\frac{1}{2}}$. We define the unbounded linear operator $A^\nu$ by $A^\nu u=\left(-\mathrm{\Delta}\right)^\nu u$ on the domain
\begin{equation*}
\mathrm{dom}\left(A^{\nu}\right):=\left\{ A^\nu u\in U:u(x)=0,\ x\in \partial D \right\},
\end{equation*}
where $\nu\ge0$. Let $\dot{U}^\nu\subset U$ denote the Banach space equipped with the inner product
\begin{equation*}
\left\langle u, v\right\rangle_{\nu}:=\sum^\infty_{i=1}\lambda_i^{\frac{\nu}{2}}\left\langle u,\phi_i(x)\right\rangle\times\lambda_i^{\frac{\nu}{2}}\left\langle v,\phi_i(x)\right\rangle
\end{equation*}
and norm
\begin{equation*}
\left\|u\right\|_{\nu}:=\left(\sum^\infty_{i=1}\lambda_i^{\nu}\left\langle u,\phi_i(x)\right\rangle^2\right)^{\frac{1}{2}},
\end{equation*}
where $\{(\lambda_i,\phi_i)\}^\infty_{i=1}$ are the eigenpairs of $-\mathrm{\Delta}$ with homogeneous Dirichlet boundary condition.

\begin{lemma}\label{le:Sec2}
{\rm(\cite{Laptev1997,Li1983,Strauss2008})} Let $\Omega\subset \mathbb{R}^d$ denote a bounded domain, $d\in\{1,2,3\}$, $|\Omega|$  be the volume of $\Omega$, and $\lambda_{i}$ the i-th eigenvalue of the Dirichlet homogeneous boundary problem for the Laplacian operator $-\mathrm{\Delta}$ in $\Omega$.
Then
\begin{equation*}
\lambda_{i}\ge \frac{4d\pi^2 }{d+2}i^{\frac{2}{d}}\cdot|\Omega|^{-\frac{2}{d}}\cdot B^{-\frac{2}{d}}_d,
\end{equation*}
where $i\in\mathbb{N}$, and $B_d$ is the volume of the unit $d$-dimensional ball.
\end{lemma}

\begin{ass}\label{as:2.1}
The function $f:U\to U$ satisfies

\begin{equation*}
f'(u)<\infty, \quad \| f'(u)- f'(v)\|\lesssim \|u-v\|, \quad u,v\in U,
\end{equation*}
and
\begin{equation*}
\|A^{\frac{\nu}{2}} f(u)\|\lesssim 1+\|A^{\frac{\nu}{2}} u\|, \quad u\in \dot{U}^\nu~{\rm with }~\nu\in[0,\gamma],
\end{equation*}
where $\gamma$ is given in Lemma \ref{lem:2}.
\end{ass}

Let $\beta(t)$ be the one-dimensional Brownian motion, and it has independent increments, i.e., the random variables $\beta(t_4)-\beta(t_3)$ and $\beta(t_2)-\beta(t_1)$ are independent whenever $0\le t_1\le t_2\le t_3\le t_4$ \cite{Morters2010}.

The infinite dimensional Brownian motion be represented as
\begin{equation*}
B(x,t)=\sum^\infty_{i=1}\sigma_{i}\beta^{i}(t)\phi_{i}(x),
\end{equation*}
where $|\sigma_{i}|\le \lambda_i^{-\rho}(\rho\ge0,\,\lambda_i$ is given in Lemma \ref{le:Sec2}$)$,  $\{\beta^{i}(t)\}_{i\in\mathbb{N}}$ are mutually independent real-valued one-dimensional Brownian motions, and $\left\{\phi_{i}(x)\right\}_{i\in\mathbb{N}}$ is an orthonormal basis of $U$. The properties of Brownian motion and definition of $B(t)$ imply $B(t)$ has ndependent increments.

We use $L^p(D,U)$ to denote the Banach space consisting of integrable random variables, that is
\begin{equation*}
\left\|u\right\|_{L^p\left(D,U\right)}:=\left(\mathrm{E}\left[\int_D\left|u(x)\right|^p\mathrm{d}x\right]\right)^{\frac{1}{p}}<\infty, \quad \ p\ge1.
\end{equation*}
Next, we define the space $L^p(D,\dot{U}^\nu)$ with norm
\begin{equation*}
\left\|u\right\|_{L^p\left(D,\dot{U}^\nu\right)}:=\left\|A^{\frac{\nu}{2}}u\right\|_{L^p\left(D,U\right)},
\quad \nu\ge0.
\end{equation*}

\section{Regularity of the solution } \label{sec:3}
In this section, we recall the regularity of mild solution of Eq. \eqref{eq:1.1}. Moreover, the H\"older continuity of the mild solution for Eq. \eqref{eq:3.2} is discussed in time. These results will be used for numerical analysis.

For the sake of brevity, we rewrite Eq. \eqref{eq:1.1} as
\begin{equation}\label{eq:3.1}
\left\{
\begin{array}{ll}
\mathrm{d} u(t)+A^\alpha  u(t)\mathrm{d}t=f\left(u(t)\right)\mathrm{d}t+\mathrm{d}B_{H,\mu}(t),& \quad \mathrm{in} \ D\times(0,T],\\
u(0)=u_0,& \quad  \mathrm{in}\ D,\\
u(t)=0,& \quad \mathrm{on}\ \partial D,
\end{array}
\right.
\end{equation}
where $u(t)=u(x,t)$ and $B(t)=B(x,t)$. There is a formal mild solution $u(t)$ for Eq.\ \eqref{eq:3.1}, that is
\begin{equation}\label{eq:3.1-1}
u(t)=S(t)u_0+\int^t_0S(t-s)f\left(u(s)\right)\mathrm{d}s+\int^t_0S(t-s)\mathrm{d}B(s),
\end{equation}
where $S(t)=\mathrm{e}^{-tA^\alpha}$.

We first consider the regularity estimate of Ornstein-Uhlenbeck process $\int^t_0S(t-s)\mathrm{d}B(s)$. Using Lemma\ref{le:Sec2} and Burkh\"older-Davis-Gundy inequality \cite{Neerven2007,Prato2014} leads to
\begin{eqnarray*}\label{eq:3.0-1}
&&\mathrm{E}\left[\left\|\int^t_0A^{\frac{\gamma}{2}}\mathrm{e}^{-(t-s)A^\alpha}\mathrm{d}B(s)\right\|^p\right]\\
&&\le C_p\left(\int^t_0\sum^\infty_{i=1}\left|\lambda_i^{\frac{\gamma-2\rho}{2}}\mathrm{e}^{-(t-s)\lambda_i^\alpha}\right|^2\mathrm{d}s\right)^{\frac{p}{2}}\nonumber\\
&&\le C_p\left(\sum^\infty_{i=1}i^{\frac{2(\gamma-\alpha-2\rho)}{d}}\right)^{\frac{p}{2}}\nonumber.
 \end{eqnarray*}

Next, we use the above inequality to derive the regularity of the solutions $u(t)$ for Eq. \eqref{eq:3.1}.
\begin{lemma}\label{lem:2}
Suppose that Assumptions \ref{as:2.1} are satisfied, $\left\|u(0)\right\|_{L^p\left(D,\dot{U}^\gamma\right)}<\infty$, $0<\epsilon<\frac{1}{2}$, $\gamma=2\rho+\alpha-\frac{d+\epsilon}{2}$, $\rho\ge0$ and $\gamma>0$. Then Eq.\ \eqref{eq:3.1} possesses a unique mild solution % $u(t)$.
\begin{equation*}
\left\|u(t)\right\|_{L^p(D,\dot{U}^\gamma)}\lesssim \frac{1}{\sqrt{\epsilon}}+\left\|u_0\right\|_{L^p\left(D,\dot{U}^\gamma\right)}
\end{equation*}
and
\begin{equation*}
\left\|u(t)-u(s)\right\|_{L^p(D,U)}\lesssim (t-s)^{\min\{\frac{\gamma}{2\alpha},\frac{1}{2}\}}\left(\frac{1}{\sqrt{\epsilon}}+\left\|u_0\right\|_{L^p\left(D,\dot{U}^\gamma\right)}\right).
\end{equation*}
\end{lemma}
\begin{proof}
Combining the triangle inequality, Eq. \eqref{eq:3.1-1} and the assumption of $f$, we have
\begin{equation*}
\begin{split}
\left\|u(t)\right\|_{L^p\left(D,\dot{U}^\gamma\right)}\lesssim& \left\|S(t)u_0\right\|_{L^p\left(D,\dot{U}^\gamma\right)}+\left\|\int^t_0S(t-s)f\left(u(s)\right)\mathrm{d}s\right\|_{L^p\left(D,\dot{U}^\nu\right)}\\
&+\left\|\int^t_0S(t-s)\mathrm{d}B(s)\right\|_{L^p\left(D,\dot{U}^\gamma\right)}\\
\lesssim& \left\|u_0\right\|_{L^p\left(D,\dot{U}^\gamma\right)}+\int^t_0\left\|u(s)\right\|_{L^p\left(D,\dot{U}^\gamma\right)}\mathrm{d}s+\frac{1}{\sqrt{\epsilon}}
\end{split}.
\end{equation*}
Then using Gr\"onwall inequality leads to
\begin{equation*}
\left\|u(t)\right\|_{L^p(D,\dot{U}^\gamma)}\lesssim \frac{1}{\sqrt{\epsilon}}+\left\|u_0\right\|_{L^p\left(D,\dot{U}^\gamma\right)}
\end{equation*}
Based on above estimate, we derive the following inequality
\begin{equation*}
\begin{split}
\left\|u(t)-u(s)\right\|_{L^p\left(D,U\right)}\lesssim& \left\|S(t-s)u(s)-u(s)\right\|_{L^p\left(D,U\right)}+\left\|\int^t_sS(t-r)f\left(u(r)\right)\mathrm{d}r\right\|_{L^p\left(D,U\right)}\\
&+\left\|\int^t_sS(t-r)\mathrm{d}B(r)\right\|_{L^p\left(D,U\right)}\\
\lesssim& (t-s)^{\min\{\frac{\gamma}{2\alpha},\frac{1}{2}\}}\left(\frac{1}{\sqrt{\epsilon}}+\left\|u_0\right\|_{L^p\left(D,\dot{U}^\gamma\right)}\right)
\end{split}.
\end{equation*}
This completes the proof of Lemma \ref{lem:2}.
\end{proof}

%To discretize Eq. \eqref{eq:3.1}, the following method  is commonly used.
%\begin{equation}\label{eq:3.0}
 %u_{n+1}-u_{n}+(t_{n+1}-t_{n}) A^\alpha  u_{n+1}=(t_{n+1}-t_{n}) f\left(u_{n}\right)+B(t_{n+1})-B(t_{n}).
%\end{equation}
In fact, the H\"older regularity of $u(t)$ shows that if using the semi-implicit Euler method discretize directly Eq. \eqref{eq:3.1}, the strong convergence rate is almost impossible to be bigger than $\min\{\frac{\gamma}{2\alpha},\frac{1}{2}\}$ in time. Therefore, in order to improve the convergence rate of time discretization, we firstly need to get an equivalent form of Eq. \eqref{eq:3.1}, the regularity of whose solution is better than the one of the solution of Eq. \eqref{eq:3.1}. Let
\begin{equation}\label{eq:3.1-2}
 u(t)=z(t)+\int^t_0S(t-s)\mathrm{d}B(s).
 \end{equation}
Then
\begin{equation}\label{eq:3.2}
\mathrm{d} z(t)+A^\alpha  z(t)\mathrm{d}t=f\left(u(t)\right)\mathrm{d}t.
\end{equation}
The formal mild solution of Eq. \eqref{eq:3.2} is as follow:
\begin{equation}\label{eq:3.3}
z(t)=S(t)z_0+\int^t_0S(t-s)f\left(u(s)\right)\mathrm{d}s.
\end{equation}

Our purpose is to obtain a higher strong convergence rate in time than $\min\{\frac{\gamma}{2\alpha},\frac{1}{2}\}$ by discretizing Eqs. \eqref{eq:3.1-2} and \eqref{eq:3.2}. Thus, we need to the following estimates to discuss the temporal error convergence.

\begin{proposition}\label{prop:4}
Suppose that Assumptions \ref{as:2.1} are satisfied, $0<\epsilon<\frac{1}{2}$, $\gamma=2\rho+\alpha-\frac{d+\epsilon}{2}$, $\gamma>0$ and $\left\|u_0\right\|_{L^2\left(D,\dot{U}^{\max\left\{3\alpha,\gamma\right\}}\right)}<\infty$, then

$\mathrm{(i)}$ for $0<\gamma\le\alpha$,
\begin{equation*}
\left\|A^{\frac{\alpha}{2}}\left(z(t)-z(s)\right)\right\|_{L^2(D,U)}\lesssim (t-s)^{\frac{1}{2}+\frac{\gamma}{2\alpha}-\epsilon}\left(\frac{1}{\epsilon\sqrt{\epsilon}}+\frac{1}{\epsilon}\left\|u_0\right\|_{L^2\left(D,\dot{U}^{\max\left\{3\alpha,\gamma\right\}}\right)}\right);
\end{equation*}

$\mathrm{(ii)}$ for $\gamma>\alpha$,
\begin{equation*}
\left\|A^{\frac{\alpha}{2}}\left(z(t)-z(s)\right)\right\|_{L^2(D,U)}\lesssim (t-s)\left(1+\left\|u_0\right\|_{L^2\left(D,\dot{U}^{\max\left\{3\alpha,\gamma\right\}}\right)}\right).
\end{equation*}
\end{proposition}
\begin{proof}
As $0<\gamma\le\alpha$, according to the regularity of $u(t)$, Lemma \ref{le:Sec2} and Eq. \eqref{eq:3.3}, we derive
\begin{equation*}
\begin{split}
&\left\|A^{\frac{\alpha}{2}}\left(z(t)-z(s)\right)\right\|_{L^2(D,U)}\\
&\lesssim  \left\|A^{\frac{\alpha}{2}}\left(S(t)-S(s)\right)z_0\right\|_{L^2(D,U)}+\left\|\int^t_sA^{\frac{\alpha}{2}}S(t-r)f\left(u(r)\right)\mathrm{d}r\right\|_{L^2(D,U)}\\
&~~~~+\left\|\int^s_0A^{\frac{\alpha}{2}}\left(S(t-r)-S(s-r)\right)f\left(u(r)\right)\mathrm{d}r\right\|_{L^2(D,U)}\\
&\lesssim  (t-s)^{\frac{1}{2}+\frac{\gamma}{2\alpha}}\left\|A^{\alpha+\frac{\gamma}{2}} u_0\right\|_{L^2(D,U)}+\int^t_s(t-r)^{-\frac{1}{2}+\frac{\gamma}{2\alpha}}\left\|A^{\frac{\gamma}{2}}f\left(u(r)\right)\right\|_{L^2(D,U)}\mathrm{d}r\\
&~~~~+\int^s_0\left\|\left(A^\alpha(s-r)\right)^{-1+\epsilon}A^{\alpha+\frac{\gamma}{2}-\alpha\epsilon}(t-s)^{\frac{1}{2}+\frac{\gamma}{2\alpha}-\epsilon}f\left(u(r)\right)\right\|_{L^2(D,U)}\mathrm{d}r\\
&\lesssim  (t-s)^{\frac{1}{2}+\frac{\gamma}{2\alpha}-\epsilon}\left(\frac{1}{\epsilon\sqrt{\epsilon}}+\frac{1}{\epsilon}\left\|A^{\frac{3\alpha}{2}} z_0\right\|_{L^2(D,U)}\right).
\end{split}
\end{equation*}
For $\alpha<\gamma<4\alpha$, similarly we have
\begin{equation*}
\begin{split}
&\left\|A^{\frac{\alpha}{2}}\left(z(t)-z(s)\right)\right\|_{L^2(D,U)}\\
&\lesssim  \left\|A^{\frac{\alpha}{2}}\left(S(t)-S(s)\right)z_0\right\|_{L^2(D,U)}+\left\|\int^t_sA^{\frac{\alpha}{2}}S(t-r)f\left(u(r)\right)\mathrm{d}r\right\|_{L^2(D,U)}\\
&~~~~+\left\|\int^s_0A^{\frac{\alpha}{2}}\left(S(t-r)-S(s-r)\right)f\left(u(r)\right)\mathrm{d}r\right\|_{L^2(D,U)}\\
&\lesssim  (t-s)\left\|A^{\frac{3\alpha}{2}} z_0\right\|_{L^2(D,U)}+\int^t_s\left\|A^{\frac{\alpha}{2}}f\left(u(r)\right)\right\|_{L^2(D,U)}\mathrm{d}r\\
&~~~~+\int^s_0\left\|\left(A^\alpha(s-r)\right)^{-1+\frac{\gamma-\alpha}{3\alpha}}A^{\frac{3\alpha}{2}}(t-s)f\left(u(r)\right)\right\|_{L^2(D,U)}\mathrm{d}r\\
&\lesssim  (t-s)\left\|A^{\max\left\{\frac{3\alpha}{2},\frac{\gamma}{2}\right\}} z_0\right\|_{L^2(D,U)}.
\end{split}
\end{equation*}
If $\gamma\ge4\alpha$, then we have the following fact£º
\begin{equation*}
\begin{split}
&\left\|A^{\frac{\alpha}{2}}\left(z(t)-z(s)\right)\right\|_{L^2(D,U)}\\
&\lesssim  (t-s)\left\|A^{\frac{3\alpha}{2}} z_0\right\|_{L^2(D,U)}+\int^t_s\left\|A^{\frac{\alpha}{2}}f\left(u(r)\right)\right\|_{L^2(D,U)}\mathrm{d}r\\
&~~~~+\left\|\int^s_0A^{\frac{3\alpha}{2}}(t-s)f\left(u(r)\right)\mathrm{d}r\right\|_{L^2(D,U)}\\
&\lesssim  (t-s)\left\|A^{\frac{3\gamma}{8}} z_0\right\|_{L^2(D,U)}+\int^t_s\left\|A^{\frac{\gamma}{8}}f\left(u(r)\right)\right\|_{L^2(D,U)}\mathrm{d}r\\
&~~~~+\int^s_0\left\|A^{\frac{3\gamma}{8}}(t-s)f\left(u(r)\right)\right\|_{L^2(D,U)}\mathrm{d}r\\
&\lesssim  (t-s)\left(1+\left\|A^{\frac{3\gamma}{8}} z_0\right\|_{L^2(D,U)}\right).
\end{split}
\end{equation*}
The proof of Proposition \ref{prop:4} is complete.
\end{proof}

In addition, we also need to get the regularity estimate of $z(t)$ in time.

\begin{proposition}\label{prop:3}
Suppose that Assumptions \ref{as:2.1} are satisfied, $\left\|u_0\right\|_{L^2\left(D,\dot{U}^{\max\left\{2\alpha,\gamma\right\}}\right)}<\infty$, $0<\epsilon<\frac{1}{2}$, $\gamma=2\rho+\alpha-\frac{d+\epsilon}{2}$ and $\gamma>0$. Then
\begin{equation*}
\left\|z(t)-z(s)\right\|_{L^2(D,U)}\lesssim (t-s)\left\|u_0\right\|_{L^2\left(D,\dot{U}^{\max\left\{2\alpha,\gamma\right\}}\right)}.
\end{equation*}
\end{proposition}
Proving Proposition \ref{prop:3} can refer to the proof of Proposition \ref{prop:4}.
\section{Temporal Discretization} \label{sec:4}
In this section, we will briefly recall the semi-implicit Euler scheme, and then modifying this scheme improve the convergence rate in time.

Taking fixed size $\tau=t_{k+1}-t_{k}$ with $k=0,1,2,\dots,\frac{T}{\tau}$. Then using the semi-implicit Euler scheme leads to the temporal semi-discretization of Eq. \eqref{eq:3.2},
\begin{equation}\label{eq:4.1-1}
\begin{split}
z_{n+1}-z_{n}+\tau A^{\alpha}z_{n+1}=&\tau f\left(u_{n}\right).
\end{split}
\end{equation}
The strong convergence rate of scheme \eqref{eq:4.1-1} is not more than $\min\{\frac{\gamma}{2\alpha},\frac{1}{2}\}$. This fact can be confirmed by the following error estimate.
\begin{equation*}
\begin{split}
&\left\|\int^{t_{n+1}}_{t_n}f\left(u(s)\right)\mathrm{d}s-\tau f\left(u_{n}\right)\right\|_{L^2(D,U)}\\
&\lesssim\left\|\int^{t_{n+1}}_{t_n}\left(f\left(u(s)\right)-f\left(u(t_n)\right)\right)\mathrm{d}s\right\|_{L^2(D,U)}+\left\|\tau\left(f\left(u(t_n)\right)- f\left(u_{n}\right)\right)\right\|_{L^2(D,U)}\\
&\lesssim\tau^{1+\min\{\frac{\gamma}{2\alpha},\frac{1}{2}\}}\left(\frac{1}{\sqrt{\epsilon}}+\left\|u_0\right\|_{L^2\left(D,\dot{U}^\gamma\right)}\right)+\tau\left\|u(t_n)- u_{n}\right\|_{L^2(D,U)}
\end{split}.
\end{equation*}
In order to obtain the higher order strong approximation of $u(t)$, we will modify the discrete scheme of $\int^{t_{n+1}}_{t_n}f\left(u(s)\right)\mathrm{d}s$ in Eq. \eqref{eq:4.1-1}. The inspiration of high accuracy discretization for $\int^{t_{n+1}}_{t_n}f\left(u(s)\right)\mathrm{d}s$  comes from Lagrange mean value theorem and independent increments of Brownian motion.
\begin{equation*}
\int^{t_{n+1}}_{t_n}\left(f\left(u(s)\right)-f\left(u(t_n)\right)\right)\mathrm{d}s=\int^{t_{n+1}}_{t_n}f'\left((1-\theta_0)u(s)+\theta_0u(t_n)\right)\left(u(s)-u(t_n)\right)\mathrm{d}s,
\end{equation*}
where $0<\theta_0<1$. Then, using Eq. \eqref{eq:3.1-2} leads to
\begin{equation}\label{eq:4.1-2}
\begin{split}
&\int^{t_{n+1}}_{t_n}\left(f\left(u(s)\right)-f\left(u(t_n)\right)\right)\mathrm{d}s\\
&=\int^{t_{n+1}}_{t_n}f'\left((1-\theta_0)u(s)+\theta_0u(t_n)\right)\left(z(s)-z(t_n)\right)\mathrm{d}s\\
&~~~~+\int^{t_{n+1}}_{t_n}f'\left((1-\theta_0)u(s)+\theta_0u(t_n)\right)\left[\int^s_0S(s-r)\mathrm{d}B(r)-\int^{t_n}_0S(t_n-r)\mathrm{d}B(r)\right]\mathrm{d}s.
\end{split}
\end{equation}
The Proposition \ref{prop:3} implies that
\begin{equation*}
\left\|\int^{t_{n+1}}_{t_n}f'\left((1-\theta_0)u(s)+\theta_0u(t_n)\right)\left(z(s)-z(t_n)\right)\mathrm{d}s\right\|_{L^2(D,U)}\lesssim\tau^2\left\|u_0\right\|_{L^2\left(D,\dot{U}^{\max\left\{2\alpha,\gamma\right\}}\right)}.
\end{equation*}
Thus, the rest of the problem is how to get the high accuracy approximation of stochastic integral in Eq. \eqref{eq:4.1-2}. We rewrite stochastic integral as
\begin{equation*}
\int^{t_{n+1}}_{t_n}\left[\int^s_{t_n}S(s-r)\mathrm{d}B(r)+\int^{t_n}_0S(s-r)\mathrm{d}B(r)-\int^{t_n}_0S(t_n-r)\mathrm{d}B(r)\right]\mathrm{d}s.
\end{equation*}
The independent increments of Brownian motion ensures that $\int^{t_{n+1}}_{t_n}\int^s_{t_n}S(s-r)\mathrm{d}B(r)\mathrm{d}s$ does not prevent us from obtaining the higher convergence rate of the modifying scheme. Therefore, we only consider the following equation.
\begin{equation*}
\begin{split}
&\int^{t_{n+1}}_{t_n}\left[\int^{t_n}_0S(s-r)\mathrm{d}B(r)-\int^{t_n}_0S(t_n-r)\mathrm{d}B(r)\right]\mathrm{d}s\\
&=A^{-\alpha}\int^{t_n}_0S(t_n-r)\mathrm{d}B(r)-A^{-\alpha}\int^{t_{n}}_0S(t_{n+1}-r)\mathrm{d}B(r)-\tau\int^{t_n}_0S(t_n-r)\mathrm{d}B(r)
\end{split}.
\end{equation*}
Using It\^{o} isometry to leads
\begin{equation*}
\mathrm{E}\left[\left(\int^{t_n}_0\mathrm{e}^{-\lambda^\alpha_i(t_n-s)}\mathrm{d}\beta^i(s)\right)^2\right]=\frac{1-\mathrm{e}^{-2\lambda_it_n}}{2\lambda^\alpha_i},
 \end{equation*}
which implies that
 \begin{equation*}
\int^{t_n}_0\mathrm{e}^{-\lambda^\alpha_i(t_n-s)}\mathrm{d}\beta^i(s)
 \end{equation*}
is a centred Gaussian random variable with variance $\frac{1-\mathrm{e}^{-2\lambda_it_n}}{2\lambda^\alpha_i}$. Then the simulation of stochastic integral $\int^{t_n}_0S(t_n-s)\mathrm{d}B(s)$ is easily implementable without approximation. Meanwhile, using the similar method get the simulation of $\int^{t_{n+1}}_{t_n}\int^{t_n}_0S(s-r)\mathrm{d}B(r)\mathrm{d}s$. To sum up, we can get a semi-discretization, whose convergence rate is better than one of Eq. \eqref{eq:4.1-1} in time. The detailed method is as follows

\begin{equation*}
\begin{split}
z_{1}-z_{0}+\tau A^{\alpha}z_{1}=&\tau f\left(u_{0}\right)
\end{split}
\end{equation*}
and for $n\ge1$,
\begin{equation}\label{eq:4.1}
\begin{split}
&z_{n+1}-z_{n}+\tau A^{\alpha}z_{n+1}\\
&=\tau f\left(u_{n}\right)-\tau \frac{f\left(u_{n}\right)-f\left(u_{n-1}\right)}{u_{n}-u_{n-1}}\int^{t_n}_0S\left(t_n-r\right)\mathrm{d}B(r)\\
&~~~~+\frac{f\left(u_{n}\right)-f\left(u_{n-1}\right)}{u_{n}-u_{n-1}}\int^{t_n}_0A^{-\alpha}\left[S\left(t_n-r\right)-S\left(t_{n+1}-r\right)\right]\mathrm{d}B(r).
\end{split}
\end{equation}
Applying the recursion gives
\begin{equation}\label{eq:4.2}
\begin{split}
z_{n+1}=&\frac{z_{0}}{\left(1+\tau A^{\alpha}\right)^{n+1}}+\tau \sum^n_{k=0}\frac{f\left(u_{k}\right)}{\left(1+\tau A^{\alpha}\right)^{n-k+1}}\\
&+\tau\sum^n_{k=1}\frac{f\left(u_{k}\right)-f\left(u_{k-1}\right)}{\left(1+\tau A^{\alpha}\right)^{n-k+1}\left(u_{k}-u_{k-1}\right)}\int^{t_k}_0S\left(t_k-r\right)\mathrm{d}B(r)\\
&-\sum^n_{k=1}\frac{f\left(u_{k}\right)-f\left(u_{k-1}\right)}{\left(1+\tau A^{\alpha}\right)^{n-k+1}\left(u_{k}-u_{k-1}\right)}\int^{t_k}_0A^{-\alpha}\left[S\left(t_k-r\right)-S\left(t_{k+1}-r\right)\right]\mathrm{d}B(r).
\end{split}
\end{equation}
Using Eq. \eqref{eq:3.1-2}, then we obtain the approximation of $u(t_n)$ in time
\begin{equation*}
 u_n=z_n+\int^{t_n}_0S(t_n-s)\mathrm{d}B(s),
 \end{equation*}
which shows
 \begin{equation*}
 u(t_n)-u_n=z(t_n)-z_n
 \end{equation*}
 and
 \begin{equation*}
 \left\|u_n\right\|_{L^2\left(D,U\right)}\lesssim\left\|z_n\right\|_{L^2\left(D,U\right)}+\left\|\int^{t_n}_0S(t_n-s)\mathrm{d}B(s)\right\|_{L^2\left(D,U\right)}.
 \end{equation*}
Thus, we analyze directly the stability and error estimates of $z_n$.

\begin{theorem}\label{th:1}
Let $z_n$ be expressed by Eq. \eqref{eq:4.2}. Suppose that Assumptions \ref{as:2.1} are satisfied, $0<\epsilon<\frac{1}{2}$, $\left\|u_0\right\|_{L^2\left(D,\dot{U}^\gamma\right)}<\infty$, $\gamma=2\rho+\alpha-\frac{d+\epsilon}{2}$, $\rho\ge0$ and $\gamma>0$. Then
\begin{equation*}
\left\|z_{n}\right\|_{L^2\left(D,U\right)}\lesssim \left\|z_{0}\right\|_{L^2\left(D,U\right)}+1.
\end{equation*}
\end{theorem}
\begin{proof}Let
\begin{equation*}
\Psi_1 u_{k}=(1-\theta)u_{k}+\theta u_{k-1},\quad  0<\theta<1.
\end{equation*}
Using triangle inequality and Lagrange mean value theorem leads to
\begin{equation*}
\begin{split}
&\left\|z_{n+1}\right\|_{L^2\left(D,U\right)}\\
&\lesssim\left\|\frac{z_{0}}{\left(1+\tau A^{\alpha}\right)^{n+1}}\right\|_{L^2\left(D,U\right)}+\tau \sum^n_{k=0}\left\|\frac{f\left(u_{k}\right)}{\left(1+\tau A^{\alpha}\right)^{n-k+1}}\right\|_{L^2\left(D,U\right)}\\
&~~~~+\tau\sum^n_{k=1}\left\|\frac{f\left(u_{k}\right)-f\left(u_{k-1}\right)}{\left(1+\tau A^{\alpha}\right)^{n-k+1}\left(u_{k}-u_{k-1}\right)}\int^{t_k}_0S\left(t_k-r\right)\mathrm{d}B(r)\right\|_{L^2\left(D,U\right)}\\
&~~~~+\sum^n_{k=1}\left\|\frac{f\left(u_{k}\right)-f\left(u_{k-1}\right)}{\left(1+\tau A^{\alpha}\right)^{n-k+1}\left(u_{k}-u_{k-1}\right)}\int^{t_k}_0A^{-\alpha}\left[S\left(t_k-r\right)-S\left(t_{k+1}-r\right)\right]\mathrm{d}B(r)\right\|_{L^2\left(D,U\right)}\\
&\lesssim\left\|z_{0}\right\|_{L^2\left(D,U\right)}+\tau \sum^n_{k=0}\left(\left\|z_{k}\right\|_{L^2\left(D,U\right)}+\left\|\int^{t_k}_0S\left(t_k-r\right)\mathrm{d}B(r)\right\|_{L^2\left(D,U\right)}\right)+1\\
&~~~~+\tau\sum^n_{k=1}\left\|f'\left(\Psi_1 u_{k}\right)\int^{t_k}_0S\left(t_k-r\right)\mathrm{d}B(r)\right\|_{L^2\left(D,U\right)}\\
&~~~~+\sum^n_{k=1}\left\|f'\left(\Psi_1 u_{k}\right)\int^{t_k}_0A^{-\alpha}\left[S\left(t_k-r\right)-S\left(t_{k+1}-r\right)\right]\mathrm{d}B(r)\right\|_{L^2\left(D,U\right)}\\
&\lesssim\left\|z_{0}\right\|_{L^2\left(D,U\right)}+\tau \sum^n_{k=0}\left(\left\|z_{k}\right\|_{L^2\left(D,U\right)}+\left\|\int^{t_k}_0S\left(t_k-r\right)\mathrm{d}B(r)\right\|_{L^2\left(D,U\right)}\right)+1\\
&~~~~+\tau\sum^n_{k=1}\left\|\int^{t_k}_0S\left(t_k-r\right)\mathrm{d}B(r)\right\|_{L^2\left(D,U\right)}+\sum^n_{k=1}\int^{t_{k+1}}_{t_k}\left\|\int^{t_k}_0S\left(r_1-r\right)\mathrm{d}B(r)\right\|_{L^2\left(D,U\right)}\mathrm{d}r_1\\
&\lesssim\tau\sum_{k=0}^{n}\mathrm{E}\left[\left\|z_{k}\right\|_{L^2\left(D,U\right)}\right]+\left\|z_{0}\right\|_{L^2\left(D,U\right)}+1.
\end{split}
\end{equation*}
Thanks to the discrete Gr\"onwall inequality, we have
\begin{equation*}
\mathrm{E}\left[\left\|z_{n+1}\right\|_{L^2\left(D,U\right)}\right]\lesssim \left\|z_{0}\right\|_{L^2\left(D,U\right)}+1.
\end{equation*}
\end{proof}

Next, we discuss the convergence behavior of $z\left(t_{n+1}\right)-z_{n+1}$ in the sense of mean-squared $L^2$-norm. Based on Proposition \ref{prop:4}, we have two cases to look at. Let $e_{n+1}=z\left(t_{n+1}\right)-z_{n+1}$. One can obtain the following estimates by using Lemma \ref{lem:2}, Proposition \ref{prop:4} and Proposition \ref{prop:3}.

\begin{theorem}\label{th:2}
Let $z_n$ be expressed by Eq. \eqref{eq:4.2}. Suppose that Assumptions \ref{as:2.1} are satisfied, $0<\epsilon<\frac{1}{2}$, $\left\|u_0\right\|_{L^2\left(D,\dot{U}^{\max\left\{3\alpha,\gamma\right\}}\right)}<\infty$, $\gamma=2\rho+\alpha-\frac{d+\epsilon}{2}$, $\rho\ge0$, $\gamma>0$ and
\begin{equation*}
 u_n=z_n+\int^{t_n}_0S(t_n-s)\mathrm{d}B(s).
 \end{equation*}
Then

$\mathrm{(i)}$ for $0<\gamma\le\alpha$
\begin{equation*}
\left\|e_{n+1}\right\|_{L^2\left(D,U\right)}\lesssim\frac{1}{\epsilon\sqrt{\epsilon}}\tau^{\frac{\gamma}{\alpha}-\epsilon}\left(\left\|u_0\right\|_{L^2\left(D,\dot{U}^{\max\left\{3\alpha,\gamma\right\}}\right)}+1\right).
\end{equation*}

$\mathrm{(ii)}$  for $\gamma>\alpha$ \begin{equation*}
\left\|e_{n+1}\right\|_{L^2\left(D,U\right)}\lesssim \tau\left(\left\|u_0\right\|_{L^2\left(D,\dot{U}^{\max\left\{3\alpha,\gamma\right\}}\right)}+1\right).
\end{equation*}
\end{theorem}

\begin{proof}
%Using inner product Eqs. \eqref{eq:3.2} and \eqref{eq:4.1}
Using directly Eqs. \eqref{eq:3.3} and \eqref{eq:4.2} to prove Theorem \ref{th:2} is complicated, thus we consider using Eqs. \eqref{eq:3.2} and \eqref{eq:4.1} to derive the above estimates. Let $\xi\in U$, then the weak formulations of Eqs. \eqref{eq:3.2} and \eqref{eq:4.1} are as follows
\begin{equation}\label{eq:4.4-1}
 \left\langle z(t_{n+1})-z(t_{n}),\xi\right\rangle+ \left\langle\int^{t_{n+1}}_{t_{n}}A^\alpha  z(t)\mathrm{d}t,\xi\right\rangle= \left\langle\int^{t_{n+1}}_{t_{n}}f\left(u(t)\right)\mathrm{d}t,\xi\right\rangle
\end{equation}
and
\begin{equation}\label{eq:4.4-2}
\begin{split}
 &\left\langle z_{n+1}-z_{n},\xi\right\rangle+\left\langle\tau A^{\alpha}z_{n+1},\xi\right\rangle\\
 &=\left\langle\tau f\left(u_{n}\right),\xi\right\rangle-\left\langle\tau \frac{f\left(u_{n}\right)-f\left(u_{n-1}\right)}{u_{n}-u_{n-1}}\int^{t_n}_0S\left(t_n-r\right)\mathrm{d}B(r),\xi\right\rangle\\
&~~~~+\left\langle\frac{f\left(u_{n}\right)-f\left(u_{n-1}\right)}{u_{n}-u_{n-1}}\int^{t_n}_0A^{-\alpha}\left[S\left(t_n-r\right)-S\left(t_{n+1}-r\right)\right]\mathrm{d}B(r),\xi\right\rangle.
\end{split}
\end{equation}
Let $\xi=e_{n+1}$
and
\begin{equation*}
\Psi_2 u(s)=(1-\theta_1)u(s)+\theta_1 u(t_k),\quad t_k<s<t_{k+1},\quad 0<\theta_1<1.
\end{equation*}
Using Lagrange mean value theorem, Eqs. \eqref{eq:4.4-1} and \eqref{eq:4.4-2} leads to
\begin{equation}\label{eq:4.4}
\begin{split}
&\left\langle e_{n+1}-e_{n},e_{n+1}\right\rangle\\
&=\left\langle-\int^{t_{n+1}}_{t_{n}}A^{\alpha}\left[z(s)-z_{n+1}\right]\mathrm{d}s+\int^{t_{n+1}}_{t_{n}}\left[f\left(u(s)\right)-f\left(u_n\right)\right]\mathrm{d}s,e_{n+1}\right\rangle\\
&~~~~-\left\langle f'\left(\Psi_1 u_{n}\right)\int^{t_n}_0A^{-\alpha}\left[S\left(t_n-r\right)-S\left(t_{n+1}-r\right)\right]\mathrm{d}B(r),e_{n+1}\right\rangle\\
&~~~~+\left\langle\tau f'\left(\Psi_1 u_{n}\right)\int^{t_n}_0S\left(t_n-r\right)\mathrm{d}B(r),e_{n+1}\right\rangle\\
&=\left\langle-\int^{t_{n+1}}_{t_{n}}A^{\alpha}\left[z(s)-z_{n+1}\right]\mathrm{d}s,e_{n+1}\right\rangle\\
&~~~~+\left\langle\int^{t_{n+1}}_{t_{n}}\left[f'\left(\Psi_2 u(s)\right)\left[u(s)-u(t_n)\right]+f\left(u(t_n)\right)-f(u_n)\right]\mathrm{d}s,e_{n+1}\right\rangle\\
&~~~~-\left\langle\int^{t_{n+1}}_{t_{n}}f'\left(\Psi_1 u_{n}\right)\int^{t_n}_0\left[S\left(s-r\right)-S\left(t_n-r\right)\right]\mathrm{d}B(r)\mathrm{d}s,e_{n+1}\right\rangle\\
\end{split}.
\end{equation}
Substituting Eq. \eqref{eq:3.1-2} into Eq. \eqref{eq:4.4}, we have
\begin{equation*}
\begin{split}
&\left\langle e_{n+1}-e_{n},e_{n+1}\right\rangle\\
&=\left\langle-\int^{t_{n+1}}_{t_{n}}A^{\alpha}\left[z(s)-z_{n+1}\right]\mathrm{d}s,e_{n+1}\right\rangle\\
&~~~~+\left\langle\int^{t_{n+1}}_{t_{n}}\left[f'\left(\Psi_2 u(s)\right)\left[z(s)-z(t_n)\right]+f\left(u(t_n)\right)-f(u_n)\right]\mathrm{d}s,e_{n+1}\right\rangle\\
&~~~~-\left\langle\int^{t_{n+1}}_{t_{n}}\left[f'\left(\Psi_1 u_{n}\right)-f'\left(\Psi_2 u(s)\right)\right]\int^{t_n}_0\left[S\left(s-r\right)-S\left(t_n-r\right)\right]\mathrm{d}B(r)\mathrm{d}s,e_{n+1}\right\rangle\\
&~~~~+\left\langle\int^{t_{n+1}}_{t_{n}}f'\left(\Psi_2 u(s)\right)\int^s_{t_{n}}S(s-r)\mathrm{d}B(r)\mathrm{d}s,e_{n+1}\right\rangle.
\end{split}.
\end{equation*}
Thanks to
\begin{equation*}
\begin{split}
\left\langle e_{n+1}-e_{n},e_{n+1}\right\rangle=\frac{1}{2}\left[\left\|e_{n+1}\right\|^2-\left\|e_{n}\right\|^2\right]+\frac{1}{2}\left\|e_{n+1}-e_{n}\right\|^2,
\end{split}
\end{equation*}
then
\begin{equation}\label{eq:4.6}
\begin{split} &\frac{1}{2}\mathrm{E}\left[\left\|e_{n+1}\right\|^2-\left\|e_{n}\right\|^2\right]+\frac{1}{2}\mathrm{E}\left[\left\|e_{n+1}-e_{n}\right\|^2\right]\\
&=\mathrm{E}\left[\left\langle-\int^{t_{n+1}}_{t_{n}}A^{\alpha}\left[z(s)-z_{n+1}\right]\mathrm{d}s,e_{n+1}\right\rangle\right]\\
&~~~~+\mathrm{E}\left[\left\langle\int^{t_{n+1}}_{t_{n}}\left(f'\left(\Psi_2 u(s)\right)\left[z(s)-z(t_n)\right]+f\left(u(t_n)\right)-f\left(u_n\right)\right)\mathrm{d}s,e_{n+1}\right\rangle\right]\\
&~~~~+\mathrm{E}\left[\left\langle\int^{t_{n+1}}_{t_{n}}f'\left(\Psi_2 u(s)\right)\int^s_{t_{n}}S(s-r)\mathrm{d}B(r)\mathrm{d}s,e_{n+1}\right\rangle\right]\\
&~~~~+\mathrm{E}\left[\left\langle\int^{t_{n+1}}_{t_{n}}\left[f'\left(\Psi_2 u(s)\right)-f'\left(u(t_n)\right)\right]\int^{t_n}_0\left[S\left(s-r\right)-S\left(t_n-r\right)\right]\mathrm{d}B(r)\mathrm{d}s,e_{n+1}\right\rangle\right]\\
&~~~~+\mathrm{E}\left[\left\langle\int^{t_{n+1}}_{t_{n}}\left[f'\left(u(t_n)\right)-f'\left(\Psi_1 u_n\right)\right]\int^{t_n}_0\left[S\left(s-r\right)-S\left(t_n-r\right)\right]\mathrm{d}B(r)\mathrm{d}s,e_{n+1}\right\rangle\right]\\
&=J_1+J_2+J_3+J_4+J_5.
\end{split}
\end{equation}
As $\gamma>0$, the Proposition \ref{prop:3} shows that the H\"older regularity of $z(t)$ is independent of $\gamma$ in time. Thus, the estimate of $J_2$ is discussed only in one case.
\begin{equation*}
\begin{split}
J_2\lesssim& \int^{t_{n+1}}_{t_{n}}\mathrm{E}\left[\left\|f'\left(\Psi_2 u(s)\right)\left[z(s)-z(t_n)\right]\right\|\left\|e_{n+1}\right\|\right]\mathrm{d}s\\
&+\int^{t_{n+1}}_{t_{n}}\mathrm{E}\left[\left\|f\left(u(t_n)\right)-f\left(u_n\right)\right\|\left\|e_{n+1}\right\|\right]\mathrm{d}s\\
\lesssim& \int^{t_{n+1}}_{t_{n}}\mathrm{E}\left[\left\|f'\left(\Psi_2 u(s)\right)\left[z(s)-z(t_n)\right]\right\|\left\|e_{n+1}\right\|\right]\mathrm{d}s+\int^{t_{n+1}}_{t_{n}}\mathrm{E}\left[\left\|e_n\right\|\left\|e_{n+1}\right\|\right]\mathrm{d}s\\
\lesssim& \int^{t_{n+1}}_{t_{n}}\mathrm{E}\left[\left\|\left[z(s)-z(t_n)\right]\right\|^2+\left\|e_{n+1}\right\|^2\right]\mathrm{d}s+\int^{t_{n+1}}_{t_{n}}\mathrm{E}\left[\left\|e_n\right\|^2+\left\|e_{n+1}\right\|^2\right]\mathrm{d}s\\
\lesssim& \tau^3\left\|u_0\right\|^2_{L^2\left(D,\dot{U}^{\max\left\{2\alpha,\gamma\right\}}\right)}+ \tau\mathrm{E}\left[\left\|e_{n}\right\|^2+\left\|e_{n+1}\right\|^2\right].
\end{split}
\end{equation*}
As $0<\gamma\le\alpha$, using the Proposition \ref{prop:4}, H\"older inequality and Young¡¯s inequality leads to the estimate of $J_1$.
\begin{equation*}
\begin{split}
J_1=&\mathrm{E}\left[\left\langle-\int^{t_{n+1}}_{t_{n}}A^{\frac{\alpha}{2}}\left[z(s)-z\left(t_{n+1}\right)+z\left(t_{n+1}\right)-z_{n+1}\right]\mathrm{d}s,A^{\frac{\alpha}{2}}e_{n+1}\right\rangle\right]\\
\lesssim& \int^{t_{n+1}}_{t_{n}}\mathrm{E}\left[\left\|A^{\frac{\alpha}{2}}\left[z(s)-z\left(t_{n+1}\right)\right]\right\|\left\|A^{\frac{\alpha}{2}}e_{n+1}\right\|\right]\mathrm{d}s-\tau\mathrm{E}\left[\left\|A^{\frac{\alpha}{2}}e_{n+1}\right\|^2\right]\\
\lesssim& \int^{t_{n+1}}_{t_{n}}\mathrm{E}\left[\left\|A^{\frac{\alpha}{2}}\left[z(s)-z\left(t_{n+1}\right)\right]\right\|^2+\left\|A^{\frac{\alpha}{2}}e_{n+1}\right\|^2\right]\mathrm{d}s-\tau\mathrm{E}\left[\left\|A^{\frac{\alpha}{2}}e_{n+1}\right\|^2\right]\\
\lesssim&\tau^{2+\frac{\gamma}{\alpha}-2\epsilon}\left(\frac{1}{\epsilon^3}+\frac{1}{\epsilon^2}\left\|u_0\right\|^2_{L^2\left(D,\dot{U}^{\max\left\{3\alpha,\gamma\right\}}\right)}\right).
\end{split}
\end{equation*}
To make use of independent increments of Brownian motion, $J_3$ is decomposed into two parts.
\begin{equation*}
\begin{split}
J_3=&\mathrm{E}\left[\left\langle\int^{t_{n+1}}_{t_{n}}\left(f'\left(\Psi_2 u(s)\right)-f'\left(u(t_n)\right)+f'\left(u(t_n)\right)\right)\int^s_{t_{n}}S(s-r)\mathrm{d}B(r)\mathrm{d}s,e_{n+1}\right\rangle\right]\\
\lesssim&\mathrm{E}\left[\int^{t_{n+1}}_{t_{n}}\left\|\left(f'\left(\Psi_2 u(s)\right)-f'\left(u(t_n)\right)\right)\int^s_{t_{n}}S(s-r)\mathrm{d}B(r)\right\|\left\|e_{n+1}\right\|\mathrm{d}s\right]\\
&+\mathrm{E}\left[\left\langle\int^{t_{n+1}}_{t_{n}}f'\left(u(t_n)\right)\int^s_{t_{n}}S(s-r)\mathrm{d}B(r)\mathrm{d}s,e_{n+1}\right\rangle\right]\\
=&J_{31}+J_{32}
\end{split}.
\end{equation*}
Combining Assumptions \ref{as:2.1}, Lemma \ref{lem:2}, Young¡¯s inequality and Burkh\"older-Davis-Gundy inequality leads to
\begin{equation*}
\begin{split}
J_{31}\lesssim&\mathrm{E}\left[\int^{t_{n+1}}_{t_{n}}\left\|\left(f'\left(\Psi_2 u(s)\right)-f'\left(u(t_n)\right)\right)\int^s_{t_{n}}S(s-r)\mathrm{d}B(r)\right\|^2\mathrm{d}s\right]+\tau\mathrm{E}\left[\left\|e_{n+1}\right\|^2\right]\\
\lesssim&\int^{t_{n+1}}_{t_{n}}\mathrm{E}\left[\left\|\left(f'\left(\Psi_2 u(s)\right)-f'\left(u(t_n)\right)\right)\right\|^4\right]\mathrm{d}s\\
&+\int^{t_{n+1}}_{t_{n}}\mathrm{E}\left[\left\|\int^s_{t_{n}}S(s-r)\mathrm{d}B(r)\right\|^4\right]\mathrm{d}s+\tau\mathrm{E}\left[\left\|e_{n+1}\right\|^2\right]\\
\lesssim&\tau^{1+\frac{2\gamma}{\alpha}}\left(\frac{1}{\epsilon}+\left\|u_0\right\|^4_{L^4\left(D,\dot{U}^\gamma\right)}\right)+\int^{t_{n+1}}_{t_{n}}\left(\sum^\infty_{i=1}\lambda^{-2\rho}_i\int^s_{t_{n}}\mathrm{e}^{-2\lambda^{\alpha}_i(s-r)}\mathrm{d}r\right)^2\mathrm{d}s+\tau\mathrm{E}\left[\left\|e_{n+1}\right\|^2\right]\\
\lesssim&\tau^{1+\frac{2\gamma}{\alpha}}\left(\frac{1}{\epsilon}+\left\|u_0\right\|^4_{L^4\left(D,\dot{U}^\gamma\right)}\right)+\tau^{1+\frac{2\gamma}{\alpha}}\left[\sum^\infty_{i=1}\lambda_i^{\gamma-\alpha-2\rho}\right]^2+\tau\mathrm{E}\left[\left\|e_{n+1}\right\|^2\right].
\end{split}.
\end{equation*}
To obtain the error estimate of $J_{32}$, we need the following auxiliary Eq. \eqref{eq:4.3}. Combining Eqs. \eqref{eq:3.3} and \eqref{eq:4.2}, we have
\begin{equation}\label{eq:4.3}
\begin{split}
e_{n+1}=&S\left(t_{n+1}\right)z_{0}-\frac{z_{0}}{\left(1+\tau A^{\alpha}\right)^{n+1}}\\
&+\sum^{n-1}_{i=0}\int^{t_{i+1}}_{t_{i}}\left(S\left(t_{n+1}-s\right)f\left(u(s)\right)-\frac{f\left(u_{i}\right)}{\left(1+\tau A^{\alpha}\right)^{n-i+1}}\right)\mathrm{d}s\\
&-\tau\sum^n_{k=0}\frac{f\left(u_{k}\right)-f\left(u_{k-1}\right)}{\left(1+\tau A^{\alpha}\right)^{n-k+1}\left(u_{k}-u_{k-1}\right)}\int^{t_k}_0S\left(t_k-r\right)\mathrm{d}B(r)\\
&+\sum^n_{k=0}\frac{f\left(u_{k}\right)-f\left(u_{k-1}\right)}{\left(1+\tau A^{\alpha}\right)^{n-k+1}\left(u_{k}-u_{k-1}\right)}\int^{t_k}_0A^{-\alpha}\left[S\left(t_k-r\right)-S\left(t_{k+1}-r\right)\right]\mathrm{d}B(r)\\
&+\int^{t_{n+1}}_{t_{n}}\left(S\left(t_{n+1}-s\right)f\left(u(t_n)\right)-\frac{f\left(u_{n}\right)}{\left(1+\tau A^{\alpha}\right)}\right)\mathrm{d}s\\
&+\int^{t_{n+1}}_{t_{n}}S\left(t_{n+1}-s\right)\left[f\left(u(s)\right)-f\left(u(t_n)\right)\right]\mathrm{d}s.
\end{split}
\end{equation}
Substituting \eqref{eq:4.3} into $J_{32}$, then independent increment of $B(t)$ implies
\begin{equation*}
\begin{split}
J_{32}=&\mathrm{E}\left[\left\langle\int^{t_{n+1}}_{t_{n}}f'\left(u(t_n)\right)\int^s_{t_{n}}S(s-r)\mathrm{d}B(r)\mathrm{d}s,\int^{t_{n+1}}_{t_{n}}S\left(t_{n+1}-s\right)\left[f\left(u(s)\right)-f\left(u(t_n)\right)\right]\mathrm{d}s\right\rangle\right]\\
\lesssim&\mathrm{E}\left[\int^{t_{n+1}}_{t_{n}}\left\|f'\left(u(t_n)\right)\int^s_{t_{n}}S(s-r)\mathrm{d}B(r)\right\|\mathrm{d}s\int^{t_{n+1}}_{t_{n}}\left\|f\left(u(s)\right)-f\left(u(t_n)\right)\right\|\mathrm{d}t\right]\\
\lesssim& \int^{t_{n+1}}_{t_{n}}\int^{t_{n+1}}_{t_{n}}\mathrm{E}\left[\left\|\int^s_{t_{n}}S(s-r)\mathrm{d}B(r)\right\|^2+\left\|u(t)-u(t_n)\right\|^2\right]\mathrm{d}t\mathrm{d}s\\
\lesssim& \tau^{2+\frac{\gamma}{\alpha}}\sum^\infty_{i=1}\lambda_i^{\gamma-\alpha-2\rho}+\tau^{2+\frac{\gamma}{\alpha}}\left(\frac{1}{\epsilon}+\left\|u_0\right\|^2_{L^2\left(D,\dot{U}^\gamma\right)}\right).
\end{split}
\end{equation*}
Combining $J_{31}$ and $J_{32}$, we obtain the estimate of $J_3$.
\begin{equation*}
\begin{split}
J_3\lesssim&\tau^{1+\frac{2\gamma}{\alpha}}\left(\frac{1}{\epsilon}+\left\|u_0\right\|^4_{L^4\left(D,\dot{U}^\gamma\right)}\right)+\tau^{1+\frac{2\gamma}{\alpha}}\left[\sum^\infty_{i=1}\lambda_i^{\gamma-\alpha-2\rho}\right]^2+\tau\mathrm{E}\left[\left\|e_{n+1}\right\|^2\right].\end{split}
\end{equation*}
Again, using Burkh\"older-Davis-Gundy inequality and H\"older inequality, we have
\begin{equation*}
\begin{split}
J_4\lesssim&\mathrm{E}\left[\int^{t_{n+1}}_{t_{n}}\left\|\left[f'\left(\Psi_2 u(s)\right)-f'\left(u(t_n)\right)\right]\int^{t_n}_0\left[S\left(s-r\right)-S\left(t_n-r\right)\right]\mathrm{d}B(r)\right\|\left\|e_{n+1}\right\|\right]\mathrm{d}s\\
\lesssim&\int^{t_{n+1}}_{t_{n}}\mathrm{E}\left[\left\|\left[f'\left(\Psi_2 u(s)\right)-f'\left(u(t_n)\right)\right]\int^{t_n}_0\left[S\left(s-r\right)-S\left(t_n-r\right)\right]\mathrm{d}B(r)\right\|^2\right]\mathrm{d}s\\
&+\tau\mathrm{E}\left[\left\|e_{n+1}\right\|^2\right]\\
\lesssim&\int^{t_{n+1}}_{t_{n}}\mathrm{E}\left[\left\|\left[f'\left(u(t_n)\right)-f'\left(\Psi_2 u(s)\right)\right]\right\|^4\right]\mathrm{d}s\\
&+\int^{t_{n+1}}_{t_{n}}\mathrm{E}\left[\left\|\int^{t_n}_0\left[S\left(s-r\right)-S\left(t_n-r\right)\right]\mathrm{d}B(r)\right\|^4\right]\mathrm{d}s+\tau\mathrm{E}\left[\left\|e_{n+1}\right\|^2\right]\\
\lesssim&\int^{t_{n+1}}_{t_{n}}\mathrm{E}\left[\left\|u(t_n)-u(s)\right\|^4\right]\mathrm{d}s\\
&+\int^{t_{n+1}}_{t_{n}}\left[\int^{t_n}_0\sum^\infty_{i=1}\lambda_i^{-2\rho}\left|\mathrm{e}^{-\lambda_i^\alpha(t_n-r)}\left(1-\mathrm{e}^{-\lambda_i^\alpha(s-t_n)}\right)\right|^2\mathrm{d}r\right]^2\mathrm{d}s+\tau\mathrm{E}\left[\left\|e_{n+1}\right\|^2\right]\\
\lesssim&\tau\mathrm{E}\left[\left\|e_{n+1}\right\|^2\right]+\tau^{1+\frac{2\gamma}{\alpha}}\left(\frac{1}{\epsilon}+\left\|u_0\right\|^4_{L^4\left(D,\dot{U}^\gamma\right)}\right)+\tau^{1+\frac{2\gamma}{\alpha}}\left[\sum^\infty_{i=1}\lambda_i^{\gamma-\alpha-2\rho}\right]^2
\end{split}
\end{equation*}
and
\begin{equation*}
\begin{split}
J_5\lesssim&\mathrm{E}\left[\int^{t_{n+1}}_{t_{n}}\left\|\left[f'\left(u(t_n)\right)-f'\left(\Psi_1 u_n\right)\right]\int^{t_n}_0\left[S\left(s-r\right)-S\left(t_n-r\right)\right]\mathrm{d}B(r)\right\|\mathrm{d}s\left\|e_{n+1}\right\|\right]\\
\lesssim&\int^{t_{n+1}}_{t_{n}}\mathrm{E}\left[\left\|f'\left(\Psi_1 u_n\right)-f'\left(\Psi_1 u(t_n)\right)+f'\left(\Psi_1 u(t_n)\right)-f'\left(u(t_n)\right)\right\|^4\right]\mathrm{d}s\\
&+\int^{t_{n+1}}_{t_{n}}\mathrm{E}\left[\left\|\int^{t_n}_0\left[S\left(s-r\right)-S\left(t_n-r\right)\right]\mathrm{d}B(r)\right\|^4\right]\mathrm{d}s+\tau\mathrm{E}\left[\left\|e_{n+1}\right\|^2\right]\\
\lesssim&\tau\mathrm{E}\left[\left\|e_{n+1}\right\|^2\right]+\tau^{1+\frac{2\gamma}{\alpha}}\left(\frac{1}{\epsilon}+\left\|u_0\right\|^4_{L^4\left(D,\dot{U}^\gamma\right)}\right)+\tau^{1+\frac{2\gamma}{\alpha}}\left[\sum^\infty_{i=1}\lambda_i^{\gamma-\alpha-2\rho}\right]^2\\
\end{split}
\end{equation*}
Then, based on the above estimates, we have
\begin{equation*}
\begin{split} &\frac{1}{2}\mathrm{E}\left[\left\|e_{n+1}\right\|^2-\left\|e_{n}\right\|^2\right]+\frac{1}{2}\mathrm{E}\left[\left\|e_{n+1}-e_{n}\right\|^2\right]\\
&\lesssim\tau^{1+\frac{2\gamma}{\alpha}-2\epsilon}\left(\frac{1}{\epsilon^3}+\frac{1}{\epsilon^2}\left\|u_0\right\|^4_{L^4\left(D,\dot{U}^{\max\left\{3\alpha,\gamma\right\}}\right)}\right)+\tau^{1+\frac{2\gamma}{\alpha}}\left[\sum^\infty_{i=1}\lambda_i^{\gamma-\alpha-2\rho}\right]^2+\tau\mathrm{E}\left[\left\|e_{n+1}\right\|^2\right],
\end{split}
\end{equation*}
which implies
\begin{equation*}
\begin{split}
&\mathrm{E}\left[\left\|e_{n+1}\right\|^2\right]\\
&\lesssim\tau^{\frac{2\gamma}{\alpha}-2\epsilon}\left(\frac{1}{\epsilon^3}+\frac{1}{\epsilon^2}\left\|u_0\right\|^4_{L^4\left(D,\dot{U}^{\max\left\{3\alpha,\gamma\right\}}\right)}\right)+\tau^{\frac{2\gamma}{\alpha}}\left[\sum^\infty_{i=1}\lambda_i^{\gamma-\alpha-2\rho}\right]^2+\tau\sum^{n}_{i=1}\mathrm{E}\left[\left\|e_{i+1}\right\|^2\right].
\end{split}
\end{equation*}
Using the discrete Gr\"onwall inequality leads to
\begin{equation*}
\begin{split}
\mathrm{E}\left[\left\|e_{n+1}\right\|^2\right]\lesssim\tau^{\frac{2\gamma}{\alpha}-2\epsilon}\left(\frac{1}{\epsilon^3}+\frac{1}{\epsilon^2}\left\|u_0\right\|^4_{L^4\left(D,\dot{U}^{\max\left\{3\alpha,\gamma\right\}}\right)}\right).
\end{split}
\end{equation*}

As $\gamma>\alpha$, by using similar steps, we can get the following estimates

\begin{equation*}
\begin{split} &\frac{1}{2}\mathrm{E}\left[\left\|e_{n+1}\right\|^2-\left\|e_{n}\right\|^2\right]+\frac{1}{2}\mathrm{E}\left[\left\|e_{n+1}-e_{n}\right\|^2\right]\\
&\lesssim\tau^{3}\left(1+\left\|u_0\right\|^4_{L^4\left(D,\dot{U}^{\max\left\{3\alpha,\gamma\right\}}\right)}\right)+\tau^{3}\left[\sum^\infty_{i=1}\lambda_i^{-2\rho}\right]^2+\tau\mathrm{E}\left[\left\|e_{n+1}\right\|^2\right].
\end{split}
\end{equation*}
Then
\begin{equation*}
\begin{split} \mathrm{E}\left[\left\|e_{n+1}\right\|^2\right]\lesssim\tau^{2}\left(1+\left\|u_0\right\|^4_{L^4\left(D,\dot{U}^{\max\left\{3\alpha,\gamma\right\}}\right)}\right).
\end{split}
\end{equation*}
This completes the proof of Theorem \ref{th:2}.
\end{proof}

\section{Numerical experiments} \label{sec:5}
In this section, we use the proposed scheme to two numerical examples for solving the Eq. \eqref{eq:1.1}. The goal is to verify the theoretical results and investigate the effect of the parameter $\alpha$ on the convergence. All numerical errors are given in the sense of mean-squared $L^2$-norm.

Based on postprocessing the stochastic integral, a modified spectral Galerkin approximation is used in the spatial direction \cite{LiuX2021}.
We solve \eqref{eq:1.1} in the one-dimensional domain $D=(0,1)$ by the proposed method with Dirichlet eigenpairs $\lambda_{j}=\pi^2j^2$, $\phi_{j}=\sqrt{2}\sin(j\pi x)$ and $j=1,2,\dots, M$. The numerical results with a smooth initial data $u(x,0)=\sin(2\pi x)$ and nonlinear term $f(u(x,t))=\sin(u(x,t))$ are presented in numerical experiments. Let $u^M_{n}$ denotes the approximation with fixed time step size $\tau=\frac{T}{N}$ at time $t=n \tau$. Using the following formula calculates the convergence rates in time:
\begin{eqnarray*}
\textrm{convergence rate}=
 \frac{\ln\left(\left\|u^M_{2N}-u^M_{N}\right\|_{L^2(D,U)}/
\left\|u^M_{N}-u^M_{N/2}\right\|_{L^2(D,U)}\right)}{\ln2}.
\end{eqnarray*}

In the numerical simulations, combining the trajectory of $u^M_{N}$ and the following equation get the approximation of $\left\|u^M_{2N}-u^M_{N}\right\|_{L^2(D,U)}$.
\begin{equation*}
\left\|u^M_{2N}-u^M_{N}\right\|_{L^2(D,U)}\approx \left(\frac{1}{K}\sum^K_{k=1}\left\|u^M_{2N,k}-u^{M}_{N,k}\right\|^2\right)^{\frac{1}{2}},
\end{equation*}
where $K=1000$, and $k$ represents the $k$-th trajectory.

For fixing $\rho=0.2$, as $\alpha = 0.4, 0.6$, and $0.8$, Theorem \ref{th:2} shows that the theoretical convergence rates are approximately $\frac{3}{4}$, $\frac{5}{6}$, and $\frac{7}{8}$, respectively. The convergence rates of the scheme \eqref{eq:4.1} are tested with $M =500$, which ensures the temporal error is the dominant one. From Table \ref{table:1}, one can see that the convergence rates of the proposed scheme are at least $\frac{\gamma}{\alpha}$ and increase with the increase of $\alpha$. The numerical results confirm the error estimate in Theorem \ref{th:2}.
\begin{table}[H]
\renewcommand\arraystretch{1.6}
\caption{Time convergence rates with $M=500$, $T=0.2$ and $\rho=0.2$.}\label{table:1}
\centering
\begin{tabular}{c c c c c c c c }
\hline
$N$ & $\alpha=0.4$ & Rate &$\alpha=0.6$ &Rate & $\alpha=0.8$ &Rate  \\
\hline
  2&0.0252&     & 0.0422&     &0.0414&          \\
  4&0.0145& 0.797& 0.0241& 0.808&0.0205&1.014\\
  8&0.0083& 0.814& 0.0136& 0.825&0.0100&1.036  \\
  \hline
\end{tabular}
\end{table}

Next, we observe the behavior of the convergence
for $\gamma>\alpha$. Choosing sufficiently big $\rho=1.2$ and $M=100$ guarantee that the dominant errors arise from the temporal approximation. Figures \ref{fig:fig1} shows that the temporal convergence rates have an order of 1 by using the proposed scheme.
\begin{figure}[htbp]
  \centering
  \includegraphics[scale=0.6]{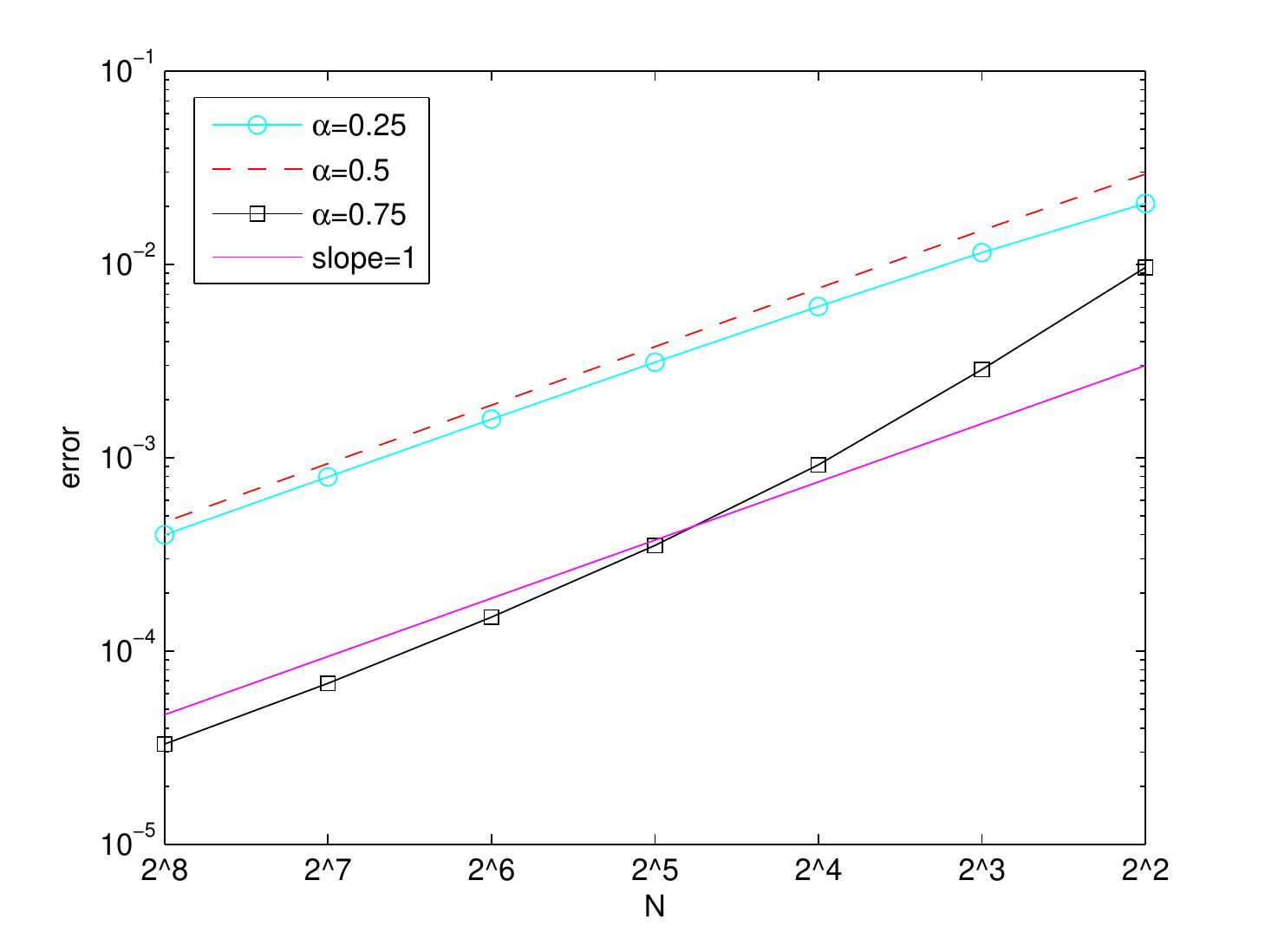}
  \caption{Temporal error convergence of scheme \eqref{eq:4.1} with $\rho=1.2$, $T=0.5$ and $M=100$.}
  \label{fig:fig1}
\end{figure}

\section{Conclusion} \label{sec:6}
In this paper, the temporal discrete scheme for the fractional stochastic PDEs is discussed. We attempt to improve the temporal convergence rate of the semi-implicit Euler scheme. The main challenge come from the H\"older regularity of $u(t)$ in time. Thus, by transforming  Eq. \eqref{eq:1.1} into an equivalent form
Eq. \eqref{eq:3.2}, whose H\"older regularity of mild solution is improved. Then, using the semi-implicit Euler scheme discretize Eq. \eqref{eq:3.2}; and a higher accuracy discretization of nonlinear term $f$ is educed by using Lagrange mean value theorem and independent increments of Brownian motion. This scheme can improve the convergence rate in time from ${\min\{\frac{\gamma}{2\alpha},\frac{1}{2}\}}$ to ${\min\{\frac{\gamma}{\alpha},1\}}$.

\section*{Acknowledgements}
The author gratefully thank the anonymous referees for valuable comments and suggestions in improving this paper.

%This work was supported by the National Natural Science Foundation of China under Grant
%No. 11671182, and the Fundamental Research Funds for the Central Universities under Grants No. lzujbky-2018-ot03.

\appendix

\end{document}